\documentclass[11pt,a4paper]{article}

\usepackage{amsmath}
\usepackage{amsfonts}
\usepackage{amssymb}
\usepackage{graphicx}
\usepackage[left=2.00cm, right=1.00cm, top=1.00cm, bottom=1.00cm]{geometry}

\usepackage{amsthm}

\newtheorem{theorem}{Theorem}

\newtheorem{lemma}{Lemma}

\author{Nasser Al-Salti and Erkinjon Karimov }

\begin{document}
\title{Inverse source problems for degenerate time-fractional PDE}
\maketitle
\begin{abstract}
In this paper, we investigate two inverse source problems for degenerate time-fractional partial differential equation in  rectangular domains. The first problem involves a space-degenerate partial differential equation and the second one involves a time-degenerate partial differential equation.  Solutions to both problem are expressed in series expansions. For the first problem, we obtained solutions in the form of Fourier-Legendre series. Convergence and uniqueness of solutions have been discussed. Solutions to the second problem are expressed in the form of Fourier-Sine series and they involve a generalized Mittag-Leffler type function. Moreover, we have established a new estimate for this generalized Mittag-Leffler type function. The obtained results are  illustrated by providing example solutions using certain given data at the initial and final time.
\end{abstract}

\medskip


\smallskip


\bigskip

\section{Introduction}
It is known that the problem of finding a solution of a partial differential equation (PDE) along with all necessary conditions is called a direct problem. Whereas, if not all characteristics (source term, coefficients or order of the PDE itself,  initial or boundary conditions) of the problem are completely known, then we deal with an inverse problem. Depending on the missing information about the characteristics of the problem, we have different types of inverse problems , such an inverse-coefficient \cite{brisar}, \cite{rund}, inverse problems of determining of order of differential equations \cite{yamli}, inverse source problems \cite{cann}, inverse problems for determining the unknown boundary condition \cite{baba} and inverse initial problem \cite{61}. The missing information will be determined using additional data. Solutions to inverse problems can be obtained through analytical and numerical methods as well. Analytical Methods include, for example, spectral method \cite{25}, the homotopy perturbation method \cite{32} and regularization method \cite{79}. For numerical methods, one may see, for example, \cite{56}, \cite{shish}. However, due to their ill-posedness and nonlinearity, inverse problems
are very challenging from a pure mathematical point of view. For more details, we refer the reader to \cite{Isakov}, \cite{yaman}.

From application point view, inverse source problems for PDEs have numerous applications in many real-life processes \cite{isak}. For instance, in heat conduction problems \cite{cann}, \cite{mal}, pollution detection problem \cite{and} and etc. Moreover, there is a growing interest in studying fractional differential equations not only due  to their importance in modeling many real-life problems, but also for their theoretical challenges form mathematical point of view. For more details about fractional differential equations and their applications, we refer the reader to \cite{kst}, \cite{p}. Hence, there is also a growing interset in studying inverse problems for time-fractional PDEs \cite{sakyam}, \cite{jinyam}, \cite{kir1}, \cite{kir2}. In these works, authors mainly dealt with PDEs without singularity. Inverse problems of space-dependent source for fractional PDEs with singularity were the subject of several works like \cite{fatma}, \cite{nass}, \cite{mich1}, \cite{mich2}. Inverse source problems have been also considered for the diffusion and sub-diffusion equations for positive operator \cite{mich3}. In a recent work by Kirane and Torebek \cite{tor}, authors investigated inverse problems for the non-local heat
equation with involution of space variable.  Inverse source problems with time-dependent source term for integer order and time-fractional diffusion equation has been also investigated, see for example, \cite{has}, \cite{wei}. Moreover, inverse problems for a class of inverse problems with a constant unknown parameter for degenerate evolution equations with the Riemann – Liouville derivative has been recently studied in \cite{fed}, where authors studied unique solvability and well-posedness issues for linear inverse problems in Banach spaces. 
		
In this paper, we consider inverse source problems for degenerate time-fractional PDEs. In particular, we consider two problems, one with space-degenerate PDE and the second with time-degenerate PDE. Inverse source problem for time-fractional space-degenerate PDE can be related to heat source identification problems where thermal conductivity in the heat conduction equation depends on space variable only, see \cite{has} for example. On the other hand, inverse source problem for time-degenerate PDE might be used in mathematical modelling of groundwater pollution as in \cite{huy}.

The rest of the paper is organized as follows. First, we present some preliminaries on fractional integrals and derivatives, Mittag-Leffler type functions and their properties and Legendre polynomials. Then, we present our main work related to the investigation of the two inverse source problems for time-fractional degenerate PDEs.

\section{Preliminaries}
\subsection{Fractional integral and differential operators}

If $\alpha \notin \mathbb{N}\cup \left\{ 0 \right\}$, the Riemann-Liouville fractional integral $I_{ax}^\alpha f$ of order $\alpha\in \mathbb{C}$ is defined by [\cite{kst},p.69]
\begin{equation}\label{fi}
I_{ax}^\alpha f(x)=\frac{1}{\Gamma(\alpha)}\int\limits_0^x \frac{f(t)dt}{(x-t)^{1-\alpha}}\,\,(x>a,\,\Re(\alpha)>0)
\end{equation}
and the Riemann-Liouville and the Caputo fractional derivatives of order $\alpha$ are defined by [\cite{kst},pp.70, 92]
\begin{equation}\label{fdrl}
_{RL}D_{ax}^\alpha f(x)=\frac{d^n}{dx^n}I_{ax}^{n-\alpha} f(x)=\frac{1}{\Gamma(n-\alpha)}\frac{d^n}{dx^n}\int\limits_0^x \frac{f(t)dt}{(x-t)^{\alpha-n+1}}\,\,(x>a,\,n=[\Re(\alpha)]+1),
\end{equation}
\begin{equation}\label{fdc}
_{C}D_{ax}^{\alpha}f(x)=\frac{1}{\Gamma \left( n-\alpha  \right)}\int\limits_{a}^{x}{\frac{{{y}^{\left( n \right)}}\left( t \right)dt}{{{\left( x-t \right)}^{\alpha -n+1}}}}\,\,\left(x>a,\, n=[\Re(\alpha)]+1,\right),
\end{equation}
respectively.

The Riemann-Liouville and the Caputo derivatives are related by the following [\cite{kst},p.91]:
\begin{equation}\label{rlcr}
_{RL}D_{ax}^{\alpha }f(x)=_{C}D_{ax}^{\alpha }f(x)+\sum\limits_{k=0}^{n-1}{\frac{{{f}^{\left( k \right)}}\left( a \right)}{\Gamma \left( k-\alpha +1 \right)}{{\left( x-a \right)}^{k-\alpha }}}\,\,
  \left( n=[\Re(\alpha)]+1,\,x>a \right).
\end{equation}
Here $\Gamma(\cdot)$ is a well-known Euler's gamma-function [\cite{kst},p.24].

\subsection{Mittag-Leffler type functions}

A two-parameter function of the Mittag-Leffler is defined by the series expansion [\cite{p}, p.17] as follows
\begin{equation}\label{ml1}
{{E}_{\alpha ,\beta }}\left( z \right)=\sum\limits_{k=0}^{\infty }{\frac{{{z}^{k}}}{\Gamma \left( \alpha k+\beta  \right)}}\,\,\left( \alpha >0,\,\beta >0 \right),
\end{equation}
which satisfies the following relation [\cite{kst},p.45] and formula of differentiation [\cite{p}, p.21], respectively
\begin{equation}\label{ml2}
{{E}_{\alpha ,\beta }}\left( z \right)-z{{E}_{\alpha ,\alpha +\beta }}\left( z \right)=\frac{1}{\Gamma \left( \beta  \right)},
\end{equation}
and
\begin{equation}\label{ml3}
{{\,}_{RL}}D_{0t}^{\gamma }\left( {{t}^{\alpha k+\beta -1}}E_{\alpha ,\beta }^{\left( k \right)}\left( \lambda {{t}^{\alpha }} \right) \right)={{t}^{\alpha k+\beta -\gamma -1}}E_{\alpha ,\beta -\gamma }^{\left( k \right)}\left( \lambda {{t}^{\alpha }} \right),
\end{equation}
where $E_{\alpha ,\beta }^{(k)}(t)=\frac{{{d}^{k}}}{d{{t}^{k}}}{{E}_{\alpha ,\beta }}(t)$, denotes the classical derivative of order $k$.

The Mittag-Leffler type function of two-parameters also satisfies the inequality presented in the following theorem.

\begin{theorem}\label{tmle}
(Theorem 1.6 in \cite{p}) If $\alpha <2$, $\beta $ is an arbitrary real number, $\mu $ is a real number such that $\pi \alpha /2<\mu <\min \{\pi ,\pi \alpha \}$ and $C$ is a real constant, then
\[
\left| {{E}_{\alpha ,\beta }}\left( z \right) \right|\le \frac{C}{1+\left| z \right|}, \,\,\,\,\left( \mu \le \arg z\le \pi  \right),\,\,\left| z \right|\ge 0.
\]
\end{theorem}

It also appears in the solution of the following Cauchy problem.

\begin{theorem}\label{cdes}
(Theorem 4.3 in \cite{kst}) Let $n-1<\alpha<n\,(n\in \mathbb{N})$ and let $0\leq \gamma<1$ bu such that $\gamma\leq \alpha$. Also let $\lambda \in \mathbb{R}$. If $f(x)\in C_\gamma[a,b]$, then the Cauchy problem
\[
\begin{array}{l}
 \left( _{C}D_{ax}^{\alpha }\,y \right)\left( x \right)-\lambda y(x)=f(x)\,\,\left(a\leq x\leq b,\, n-1<\alpha<n;\, n\in \mathbb{N},\, \lambda \in \mathbb{R}\right),\\
 y^{(k)}(a)=b_k\,\,\left(b_k\in \mathbb{R};\, k=0,1,2,..., n-1\right)\\
\end{array}
\]
has a unique solution $y(x)\in C_\gamma^{\alpha, n-1}[a,b]$ and this solution is given by
\begin{equation*}
y(x)=\sum\limits_{j=0}^{n-1} b_j (x-a)^j E_{\alpha, j+1}\left[\lambda (x-a)^\alpha\right]+\int\limits_0^x (x-t)^{\alpha-1}E_{\alpha, \alpha}\left[\lambda (x-t)^\alpha\right]f(t)dt.
\end{equation*}
\end{theorem}
Here $C_\gamma^{\alpha, n-1}[a,b]=\left\{y(x)\in C^{n-1}[a,b],\,\,\left({}_CD_{ax}^\alpha y\right)(x)\in C_\gamma[a,b]\right\}.$

Now, we recall following generalized Mittag-Leffler type function, which was introduced by Kilbas \cite{ks}:
\[
E_{\alpha, m, n}(z)=1+\sum\limits_{k=1}^\infty\prod_{j=0}^{k-1}\dfrac{\Gamma(\alpha(jm+n)+1)}{\Gamma(\alpha(jm+n+1)+1)}z^k,
\]
where $\alpha,\,n\in \mathbb{C}$, $m\in \mathbb{R}$ such that $\Re(\alpha)>0,\,m>0,\,\alpha(jm+n)\notin \mathbb{Z}^- \,(j\in\mathbb{N}_0)$.

As a particular case, if $m=1$, we have
\[
E_{\alpha,1,n}(z)=\Gamma(\alpha n+1)E_{\alpha, \alpha n+1}(z).
\]

We have established a new estimate for this generalized Mittag-Leffler type function.  This result is presented in the following lemma:

\begin{lemma}
If $\alpha n+1>\bar{c}$, $m,\alpha>0$ and $|z|<1$, then
\[
\left|E_{\alpha, m, n}(z)\right|\leq \dfrac{1}{1-|z|},
\]
where $\bar{c}$ is a point of minimum of $\Gamma(x)$ at $x>0$, precisely, $1<\bar{c}<2$.
\end{lemma}
\begin{proof}

Assume that $\alpha[(k-1)m+n]+1>\bar{c}$ for $k=1,2,3,...$ If we choose $\alpha$ and $n$ such that $\alpha n+1>\bar{c}$, then above given inequality will be true for all $k=1,2,3,...$ That condition guarantees that
\[
\dfrac{\Gamma[\alpha((k-1)m+n)+1]}{\Gamma[\alpha((k-1)m+n+1)+1]}<1.
\]

Hence,
\[
\left|E_{\alpha, m, n}(z)\right|\leq 1+|z|+|z|^2+|z|^3+...=\sum\limits_{k=0}^{\infty}|z|^k.
\]
If $|z|<1$, then according to the geometric series, latter series will converge to $\frac{1}{1-|z|}$. This completes the proof of the Lemma.
\end{proof}

\begin{lemma}[\cite{mlb}, p.107]
If $\alpha,\, m$ and $n$ are real numbers such that the condition
\[
\alpha>0,\,m>0,\,\alpha (im+n)+1\neq -1,-2,-3,...\,(i=0,1,2,3,...)
\]
is satisfied, then $E_{\alpha, m, n}(z)$ is an entire function of variable $z$. 

\end{lemma}
\subsection{Legendre polynomials}
The following Legendre equation
\begin{equation}\label{le}
(1-x^2)y''(x)-2xy'(x)+\lambda y(x)=0
\end{equation}
has bounded solution in $[-1,1]$ only if $\lambda=n(n+1),\,\,n=0,1,2,...$ and it has a form
\begin{equation*}
y(x)=P_n(x)=\frac{1}{2^n \cdot n!}\frac{d^n(x^2-1)^n}{dx^n}\,\,\,\,(n=0,1,2,...),
\end{equation*}
where $P_n(x)$ is Legendre polynomial \cite{Kap}.

Below we give some statements considering certain properties of this polynomial for the sake of the reader.

\begin{theorem}[\cite{Kap}, p.508]
$P_n(x)$ ia a polynomial of degree $n$. $P_n(x)$ is an odd function or even function according to whether $n$ is odd or even. The following identities hold for $n=1,2,...$

a) $P_n'(x)=xP_{n-1}'(x)-nP_{n-1}(x)$;

b) $P_n(x)=xP_{n-1}(x)+\frac{x^2-1}{n}P_{n-1}'(x)$.
\end{theorem}

\begin{theorem}[\cite{Kap}, p.509]
The Legendre polynomials satisfy the following identities and relations:

c) $P_{n+1}'(x)-P_{n-1}'(x)=(2n+1)P_n(x)\,\,\,(n\geq 1)$;

d) $\frac{d}{dx}\left[(1-x^2)P_n'(x)\right]+n(n+1)P_n(x)=0$;

e) $P_{n+1}(x)=\frac{(2n+1)xP_n(x)-nP_{n-1}(x)}{n+1}\,\,(n\geq 1)$;

f) $P_n(1)=1,\,\,P_n(-1)=(-1)^n$;

g) $\frac{1-x^2}{n}(P_n')^2+P_n^2=\frac{1-x^2}{n}(P_{n-1}')^2+P_{n-1}^2\,\,(n\geq 1)$;

h) $\frac{1-x^2}{n}(P_n')^2+P_n^2\leq 1\,\,\,(n\geq 1,\,\, |x|\leq 1)$;

i) $|P_n(x)|\leq 1 \,\,(|x|\leq 1)$;

j) $\int\limits_{-1}^1P_n(x)P_m(x)dx=0\,\,(n\neq m)$;

k) $\int\limits_{-1}^1[P_n(x)]^2dx=\frac{2}{2n+1}$;

l) $x^n$ can be expressed as a linear combination of $P_0(x),P_1(x),...,P_n(x)$.
\end{theorem}

\begin{theorem}[\cite{Kap}, p.509]
The Legendre polynomials $P_n(x)\, (n=0,1,2,...)$ form an orthogonal system in the interval $-1\leq x \leq 1$, and 
$\|P_n(x)\|^2=\frac{2}{2n+1}$.
\end{theorem}
Any an arbitrary piecewise continuous function $f$ in $-1\leq x \leq 1$ can be expressed in the form of a Fourier series  with respect to the system $\{P_n(x)\}$:
\begin{equation}\label{flc}
\sum\limits_{n=0}^\infty c_nP_n(x),\,\,\,c_n=\frac{(f,P_n)}{\|P_n\|^2}=\frac{2n+1}{2}\int\limits_{-1}^1f(x)P_n(x)dx,
\end{equation}
which is called Fourier-Legendre series.

\begin{theorem}[\cite{Kap}, p.511]\label{uc}
If $f(x)$ is very smooth for $-1\leq x \leq 1$, then the Fourier-Legendre series of $f(x)$ converges uniformly to $f(x)$ for $-1\leq x \leq 1$. 
\end{theorem}

Note that there is no condition of periodicity or any other condition is imposed on $f(x)$ at $x=\pm 1$, as in analogous to trigonometric series. This is due to the symmetry of Legendre polynomials. Finally, we recall the following theorem regarding Legendre polynomials:

\begin{theorem}[\cite{Kap}, p.511]\label{os}
The Legendre polynomials form a complete orthogonal system for the interval $[-1,1]$.
\end{theorem}

\section{Inverse source problem for space-degenerate PDE}
\subsection{Problem Formulation}
We formulate an inverse source problem for space-degenerate partial differential equation with the time-fractional Caputo derivative as follows:

\textbf{Problem 1.} Find a pair of functions $\{U(t,x), h(x)\}$, which satisfies the equation 
\begin{equation}\label{Eq}
{}_CD_{0t}^\alpha U(t,x)=\left[(1-x^2)U_x\right]_x+h(x)
\end{equation}
in the domain $\Omega=\left\{(t,x):\,\,-1<x<1,\,0<t<T\right\}$ together with the initial condition
\begin{equation}\label{Ic}
U(0,x)=v(x),\,\,-1\leq x \leq 1,
\end{equation}
and the over-determining condition
\begin{equation}\label{Ov}
U(T,x)=w(x),\,\, -1\leq x \leq 1,
\end{equation}
such that $U, U_x$ are bounded at $x=-1, x=1$. Here, $0<\alpha<1$, $T>0$ and $v(x), w(x)$ are given functions. 

\subsection{Existence and Convergence of the Solution}
Solving the homogeneous equation corresponding to (\ref{Eq}) using separation of variables leads to the Legendre equation (\ref{le}). According to Theorem \ref{os}, the Legendre polynomials form a complete orthogonal system in $[-1, 1]$, hence the solution set $\{U(t,x), h(x)\}$ and the given data $v(x), w(x)$ can be represented in a form of Fourier-Legendre series as follow:
\begin{eqnarray}
&& U(t,x)= \sum\limits_{n=0}^\infty U_{n}(t)P_n(x), \quad h(x)= \sum\limits_{n=0}^\infty  h_{n}P_n(x), \label{Uhn} \\
&& v(x)= \sum\limits_{n=0}^\infty  g_{n}P_n(x), \quad w(x)= \sum\limits_{n=0}^\infty  w_{n}P_n(x)\label{vwn},
\end{eqnarray}
where the coefficients $U_{n}(t)$, $h_n$ are the unknowns to be found and the coefficients $v_{n}$, $w_{n}$, according to (\ref{flc}), are given by
$$
v_n=\frac{2n+1}{2}\int\limits_{-1}^{1}{v(x)P_n(x)dx},
\quad w_n=\frac{2n+1}{2}\int\limits_{-1}^{1}{w(x)P_n(x)dx}, \quad n=0,1,2,...
$$
Now, substituting the above series representations into (\ref{Eq}) - (\ref{Ov}), we obtain the following equation for $U_n(t), h_n$:
\begin{equation} \label{Eq_Uhn}
{}_C D^{\alpha}_{0t}U_n(t)+\lambda_n U_n(t)= h_n  \quad n=0,1,2,...
\end{equation}
subjected to the following conditions:

\begin{equation} \label{Cn_Uhn}
U_n(0)=v_n, \quad U_n(T)=w_n, \quad n=0,1,2,...
\end{equation}
The solutions of equation (\ref{Eq_Uhn}) are found to be (See Theorem 2)
\begin{eqnarray*}
&&U_0(t)= c_0 + \dfrac{h_0}{\Gamma(\alpha +1)} t^\alpha \\
&& \\
&&U_{n}(t)=c_{n}E_\alpha \left(-\lambda_{n} t^\alpha \right) + \dfrac{h_{n}}{\lambda_{n}}, \quad n=1,2,...
\end{eqnarray*}
where $c_{0}$, $c_{n}$ are unknown constants. Using the conditions (\ref{Cn_Uhn}), we get
\begin{eqnarray*}
&&c_0=v_0, \hspace{3cm} h_0=\dfrac{\Gamma(\alpha+1)}{T^\alpha}(w_0 - v_0)\\
&& \\
&&c_{n}=\dfrac{v_{n}-w_{n}}{1-E_\alpha \left(-\lambda_{n} T^\alpha \right)}, \quad  \quad h_{n}=\lambda_{n}\left(v_{n}-c_{n}\right),  \quad n=1,2,...
\end{eqnarray*}
Hence, we have 
\begin{eqnarray*}
&&U_0(t)= v_0 + (w_0 - v_0)\left(\dfrac{t}{T}\right)^\alpha, \quad U_{n}(t)=\dfrac{1-E_\alpha \left(-\lambda_{n} t^\alpha \right) }{1-E_\alpha \left(-\lambda_{n} T^\alpha \right)} \left(w_{n}-v_n \right) + v_{n}, \quad n=1,2,... \\
&& \\
&&h_0=(w_0-v_0)\dfrac{\Gamma(\alpha +1)}{T^\alpha}, \quad \quad h_{n}=\dfrac{\lambda_n}{1-E_\alpha \left(-\lambda_{n} T^\alpha \right)} \left(w_{n}-v_n \right) + \lambda_n v_{n}, \quad n=1,2,...
\end{eqnarray*}
Substituting back, we get the following expressions for $U(t,x)$ and $h(x)$:

$$
U\left( {t,x} \right)= \dfrac{t^{\alpha}}{T^{\alpha}}(w_0 - v_0)+ v \left(x \right)
 +\sum\limits_{n= 1}^\infty \dfrac{1-E_\alpha \left(-\lambda_{n} t^\alpha \right)}{1-E_\alpha \left(-\lambda_{n} T^\alpha \right)}(w_{n} - v_{n}) P_n(x),
$$
and
$$
h\left( x \right)= \dfrac{\Gamma(\alpha+1)}{T^{\alpha}}(w_0 - v_0) - \dfrac{d}{dx}\left[\left(1-x^2\right)g'(x)\right]
 +\sum\limits_{n
= 1}^\infty  \lambda_{n} \dfrac{w_{n} - v_{n}}{1-E_\alpha \left(-\lambda_{n} T^\alpha \right)} P_n(x).
$$

In order to complete the proof of existence of a formal solution, we need to prove the uniform convergence of the series appearing in the above expressions for $U(t,x), h(x)$ as well as the corresponding series in $[(1-x^2)U_x]_x$ and ${}_{C}D_{0t}^\alpha U(t,x)$. We start with $U(t,x), h(x)$ by considering the following estimates:
$$\left|U\left( {t,x} \right)\right| \le \left|w_0\right| + \left|v_0\right|+ \left|v \left(x \right)\right|
 +C_1\sum\limits_{n= 1}^\infty \left( \left|w_n\right| + \left|v_n\right| \right),$$
and
$$
\left|h\left( x \right)\right| \le \dfrac{\Gamma(\alpha+1)}{T^{\alpha}}\left(\left|w_0\right| + \left|v_0\right|\right) +\left| \left[\left(1-x^2\right)v'(x)\right]'\right|
 +C_2\sum\limits_{n= 1}^\infty n(n+1)\left( \left|w_n\right| + \left|v_n\right| \right),$$
where $C_1, C_2$ are positive constants. Here, we have used the following properties of the Mitag-Leffler function and the Legendre polynomials:
\[ E_{\alpha,\beta}(\lambda t^{\alpha}) \le M, \quad  \quad 0<\alpha\le \beta \le 1, \ 0\le t \le T < \infty,  \]
\[\left|P_n(x)\right|\le 1, \quad \quad \left|x \right| \le 1,\]
for some positive constant $M$. Clearly the convergence of the above series depends on finding appropriate estimates for $v_n$ and $w_n$.  This can be done by utilizing the following properties of the Legendre polynomials:
\[ (2n+1) P_n = P'_{n+1} - P'_{n-1} \]
\[ P_n(1)=1, \quad \quad P_n(-1)= (-1)^n, \]
and using integration by parts to get
\begin{equation*}
v_n=-\frac{1}{2}\int\limits_{-1}^{1}{g'(x)\left(P_{n+1}(x)-P_{n-1} \right)dx} 
= -\frac{1}{2} \left[ (v',P_{n+1}) - (v',P_{n-1})\right].
\end{equation*} 
Hence, we have
\begin{eqnarray*}
|v_n|\le  \frac{1}{2} \left[ \left| (v',P_{n+1})\right| + \left|(v',P_{n-1})\right| \right] 
\le  \frac{1}{2} \left( || v'|| \cdot ||P_{n+1}|| + || v'|| \cdot ||P_{n-1}|| \right)\\
\le   \frac{1}{2}|| v'|| \left( \dfrac{\sqrt{2}}{(2n+3)^{\frac{1}{2}}} + \dfrac{\sqrt{2}}{(2n-1)^{\frac{1}{2}}} \right)
\le   \frac{|| v'||\sqrt{2}}{(2n-1)^{\frac{1}{2}}}.
\end{eqnarray*}

Here, we have used the Schwartz inequality $\left|(f,g)\right| \le || f|| \cdot ||g|| $ and $||P_n(x)||^2 = \dfrac{2}{2n+1}$. Repeating the above process one more time, one will arrive to
\[ |v_n| \le \frac{4\sqrt{2}}{(2n-3)^{\frac{3}{2}}}|| v''||. \]

A similar estimate can also be obtained for $w_n$. Clearly, these estimates on using Weierstrass M-test would be enough to ensure the uniform convergence of the series representation of $U(t,x)$ provided $v, w \in C(-1, 1)$ and $v'', w'' \in L(-1, 1)$. However, the convergence of the series representation of $h(x)$ would require integration by parts two more times since the series includes the term $n(n+1)$. This would lead to the following estimates for $v_n, w_n$:
\begin{equation} \label{P7_est}
|f_n| \le \frac{6\sqrt{2}}{(2n-7)^{\frac{7}{2}}}|| f^{(4)}||, \quad \quad f=v, w. 
\end{equation} 
Hence, the series representation of $h(x)$ converges uniformly provided $v, w \in C^3(-1, 1)$ and $v^{(4)}, w^{(4)} \in L(-1, 1)$. Finally, the series representations of ${}^{C}D_{0t}^\alpha U(t,x)$ and $[(1-x^2)U_x]_x$ are given by 

$$
{}^{C}D_{0t}^\alpha U\left( {t,x} \right)= \dfrac{\Gamma (\alpha +1)}{T^{\alpha}}(w_0 - v_0)
 +\sum\limits_{n= 1}^\infty \dfrac{ n(n+1) E_\alpha \left(-\lambda_{n} t^\alpha \right)}{1-E_\alpha \left(-\lambda_{n} T^\alpha \right)}(w_{n} - v_{n}) P_n(x),
$$
and 
$$[(1-x^2)U_x]_x= \left[ \left(1-x^2\right) v'(x)\right]'  -\sum\limits_{n= 1}^\infty \dfrac{1-E_\alpha \left(-\lambda_{n} t^\alpha \right)}{1-E_\alpha \left(-\lambda_{n} T^\alpha \right)}(n(n+1))(w_{n} - v_{n}) P_n(x),$$
and they have the following estimates
$$|{}^{C}D_{0t}^\alpha U\left( {t,x} \right)| \le \dfrac{\Gamma(\alpha+1)}{T^{\alpha}}\left(\left|w_0\right| + \left|v_0\right|\right)+C_3\sum\limits_{n= 1}^\infty n(n+1)\left( \left|w_n\right| + \left|v_n\right| \right), $$

and

$$| [(1-x^2)U_x]_x | \le |\left[ \left(1-x^2\right) g'(x)\right]'| +C_4\sum\limits_{n= 1}^\infty n(n+1)\left( \left|w_n\right| + \left|g_n\right| \right), $$
where $C_3, C_4$ are positive constants. Hence, the estimate (\ref{P7_est}) along with the corresponding conditions on $v$ and $w$ would be enough to ensure the uniform convergence of the series representations of ${}^{C}D_{0t}^\alpha U(t,x)$ and $[(1-x^2)U_x]_x$.

\subsection{Uniqueness of the Solution}

Suppose that there are two solution sets $\left\{ {u_1 \left( {t,x} \right),h_1 \left( x \right)} \right\}$ and $\left\{ {u_2 \left( {t,x} \right), h_2\left( x \right)} \right\}$ to the inverse problem (\ref{Eq})- (\ref{Ov}). Denote $$U\left( {t,x} \right) = u_1 \left( {t,x} \right) - u_2 \left( {t,x} \right),$$
and
$$h\left( x \right) = h_1 \left( x \right) - h_2 \left( x
\right).$$ Then, the functions $U\left( {x,t} \right)$ and $h\left(
x \right)$ clearly satisfy equation (\ref{Eq}) and the homogeneous conditions
\begin{equation} \label{H_IC}
U\left({0,x} \right) = 0, \quad U\left( {T,x} \right) =
0, \quad x \in \left[ { -1, 1 } \right]
\end{equation}
Let us now introduce the following
\begin{equation}\label{u1k}
U_{n} \left(t \right) =\int\limits_{ - 1 }^1 {U\left( {t,x} \right) P_n(x)\,dx}, \quad n=0,1,2,...,
\end{equation}
\begin{equation}\label{f1k}
h_{n} = \int\limits_{ -1 }^1 {h\left( x \right)P_n(x)\, dx},  \quad n=0,1,2,...,
\end{equation}

Note that the homogeneous conditions in (\ref{H_IC}) lead to
\begin{equation} \label{Un_HC}
U_{n}(0)= U_{n}(T)=0, \quad n=1,2,...,
\end{equation}
Differentiating equation (\ref{u1k}) gives
$${}^{C}D_{0t}^\alpha U_{n} \left( t
\right) = \int\limits_{ -1}^1 {\left[ \left(1-x^2 \right)U_x \right]_x P_n(x) \, dx}  + h_{n},$$
which on integrating by parts twice reduces to
$${}^{C}D_{0t}^\alpha U_{n} \left( t
\right) +n(n+1)U_n(t)=h_n.$$
One can then easily show that this equation together with the conditions in (\ref{Un_HC}) imply that
 $$h_{n}  = 0, \quad u_{n} \left( t \right) \equiv 0.$$
Therefore, due to the completeness of the system $\left\{P_n(x)\right\}$ in $ \left[ -1, 1 \right]$, we must have
 $$ h\left( x \right) \equiv 0, \quad U\left(
t,x \right) \equiv 0,  \quad x \in [-1, 1], \quad t \in [0, T], $$
which ends the proof of uniqueness.

\subsection{Main Result}
The main result for the inverse problem (\ref{Eq}) - (\ref{Ov}) can be summarized in the following theorem:

\begin{theorem}\label{th1} Let $v, w \in C^3(-1, 1)$ and $v^{(4)}, w^{(4)} \in L_2(-1, 1)$. Then, a unique solution to
the \ inverse problem (\ref{Eq}) - (\ref{Ov}) exists and it can be written in the form
$$
U\left( {t,x} \right)= \dfrac{t^{\alpha}}{T^{\alpha}}(w_0 - v_0)+ v \left(x \right)
 +\sum\limits_{n= 1}^\infty \dfrac{1-E_\alpha \left(-\lambda_{n} t^\alpha \right)}{1-E_\alpha \left(-\lambda_{n} T^\alpha \right)}(w_{n} - v_{n}) P_n(x),
$$
and
$$
h\left( x \right)= \dfrac{\Gamma(\alpha+1)}{T^{\alpha}}(w_0 - v_0) - \dfrac{d}{dx}\left[\left(1-x^2\right)v'(x)\right]
 +\sum\limits_{n
= 1}^\infty  \lambda_{n} \dfrac{w_{n} - v_{n}}{1-E_\alpha \left(-\lambda_{n} T^\alpha \right)} P_n(x).
$$
where $\lambda_n=n(n+1)$ and 
$$
v_n=\frac{2n+1}{2}\int\limits_{-1}^{1}{v(x)P_n(x)dx},
\quad w_n=\frac{2n+1}{2}\int\limits_{-1}^{1}{w(x)P_n(x)dx}, \quad n=0,1,2,...
$$
\end{theorem}

This result will be illustrated by a simple example solution in the next section.

\subsection{Example Solution}

Here we consider the following choices for the functions $v$  and $w$:
$$v(x)=0, \quad \text{and} \quad w(x)=w_0+w_2(3x^2-1).$$
Solutions Corresponding to this choice of conditions are given by
$$U(x,t)= w_0\dfrac{t^\alpha}{T^\alpha}+ w_2 \dfrac{1- E_\alpha(-6 t^\alpha)}{1- E_\alpha(-6 T^\alpha)}(3x^2-1) \quad \text{and} \quad h(x)=w_0\dfrac{\Gamma(\alpha+1)}{T^\alpha}+  \dfrac{6w_2(3x^2-1)}{1- E_\alpha(-6 T^\alpha).}$$
These solutions are illustrated in Figures (1) - (2) for $T=1.$ Figure (\ref{fig1}) shows the solution profile at different times and the source function for a fixed value of the fractional order. The effect of the over-determining condition is clearly seen in the solution profile and in the shape of the source function. Moreover, the solution is increasing with time and reaching its maximum when $t = T = 1 $. The effect of the order of the fractional derivative $\alpha$ is illustrated in Figure (\ref{fig2}). It shows very small effect in the source term, while its effect is more apparent in the solution profile. It shows that the solution is decreasing with the increase of the fractional order.
\begin{figure}
\includegraphics[height=7cm,width=7cm]{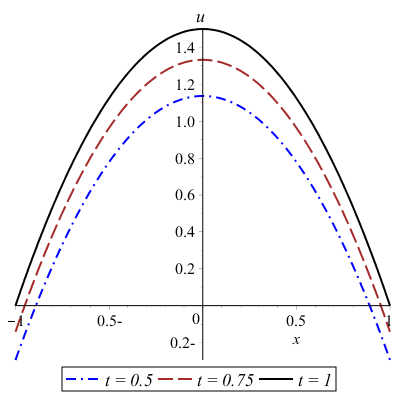}
\includegraphics[height=7cm,width=7cm]{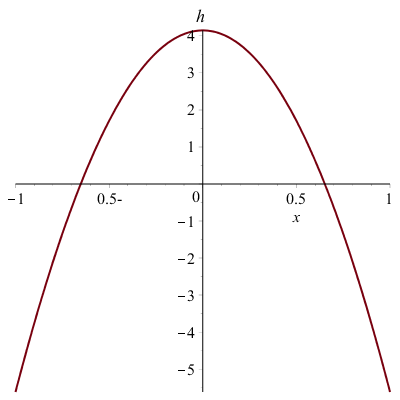}
\caption{Graphs of $U(x,t)$ at different times (left) and $h(x)$ (right) for $\alpha=0.6$. }
\label{fig1}
\end{figure}

\begin{figure}
\includegraphics[height=7cm,width=7cm]{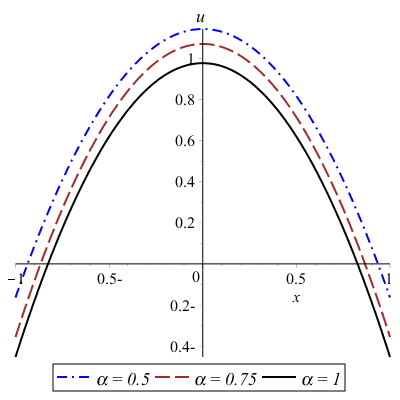}
\includegraphics[height=7cm,width=7cm]{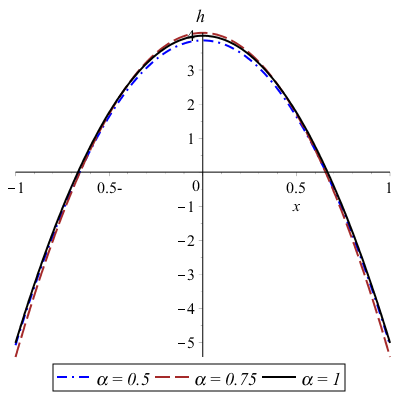}
\caption{Graphs of $u(x,t)$ at $t=0.5$ (left) and $f(x)$ (right) for different values of $\alpha$.}
\label{fig2}
\end{figure}

\section{Inverse source problem for time-degenerate PDE}
In this section, we investigate the following time-degenerate partial differential equation with time-fractional Riemann-Liouville  derivative:
\begin{equation}\label{tdme}
_{RL}D_{0t}^\alpha \bar{u}(t,x)=t^\beta \bar{u}_{xx}(t,x)+\bar{h}(x),
\end{equation}
where $\alpha,\beta\in \mathbb{R}$ such that $0<\alpha<1,\, \beta\geq 0$.

In particular, we consider the following inverse source problem: 

\textbf{Problem 2.} Find a pair of functions $\left\{\bar{u}(t,x), \bar{h}(x)\right\}$ which is a regular solution for the equation (\ref{tdme}) in the domain $\Phi=\left\{(t,x):\,0<x<1,\, 0<t<\bar{T}\right\}$ satisfying the boundary conditions
\begin{equation}\label{tdbc}
\bar{u}(t,0)=0,\,\,\,\bar{u}(t,1)=0,\,0\leq t\leq\bar{T}
\end{equation}
together with the initial condition
\begin{equation}\label{tdic}
\left.I_{0t}^{1-\alpha}\bar{u}(t,x)\right|_{t=0}=\bar\varphi(x),\,\,0\leq x\leq 1
\end{equation}
and the over-determining condition
\begin{equation}\label{tdoc}
\bar{u}(\bar{T},x)=\bar{\psi}(x),\,\,0\leq x \leq 1.
\end{equation}

Here $I_{0t}^{1-\alpha} (\cdot)$ is the Riemann-Liouville fractional integral of order $1-\alpha$, $\bar{\varphi }(x)$ and $\bar{\psi}(x)$ are given functions such that $\bar{\varphi}(0)=\bar{\varphi}(1)=\bar{\psi}(0)=\bar{\psi}(1)=0$.

We use the method of separation of variables for the corresponding homogeneous equation of (\ref{tdme}) and obtain the following boundary problem with respect to space-variable
\begin{equation}\label{tdsl}
X''(x)-\lambda X(x)=0,\,\,X(0)=0,\,X(1)=0,\,\,0\leq x \leq 1,
\end{equation}
which is a well-known Sturm-Liouville eigenvalue problem and it has a complete orthonormal system of eigenfunctions $X_n(x)=\sin n\pi x,\,n=1,2,...$

We now look for  solution to Problem 2 in the following form:

\begin{equation}\label{tdrs}
\bar{u}(t,x)=\sum\limits_{k=1}^\infty \bar{u}_k(t)\sin k\pi x,\,\,\,\bar{h}(x)=\sum\limits_{k=1}^\infty \bar{h}_k\sin k\pi x,
\end{equation}
where $\bar{u}_k$ and $\bar{h}_k$ are unknown coefficients.

Substituting (\ref{tdrs}) into (\ref{tdme}), we will get the following ordinary differential equation of fractional order with respect to time-variable:
\begin{equation}\label{tdfe}
{}_{RL}D_{0t}^\alpha \bar{u}_k(t)-(k\pi)^2t^\beta\bar{u}_k(t)=\bar{h}_k.
\end{equation}
Initial condition (\ref{tdic}) will be given as
\begin{equation}\label{tdfic}
\left.I_{0t}^{1-\alpha}\bar{u}_k(t)\right|_{t=0}=\bar{\varphi}_k,
\end{equation}
where $\bar{\varphi}_k=2\int\limits_0^1\bar{\varphi}(x)\sin k\pi x dx$ is a Fourier coefficient of the function $\bar{\varphi}(x)$. Similarly, over-determining condition (\ref{tdoc}) leads us to
\begin{equation}\label{tdfoc}
\bar{u}_k(\bar{T})=\bar{\psi}_k
\end{equation}
with $\bar{\psi}_k=2\int\limits_0^1\bar{\psi}(x)\sin k\pi x dx$. 

Solution of the Cauchy problem (\ref{tdfe})-(\ref{tdfic}) can be written as \cite{kst} (see page 247)
\begin{equation}\label{tdfs}
\bar{u}_k(t)=\bar{\varphi}_kt^{\alpha-1}E_{\alpha,1+\frac{\beta}{\alpha}, 1+\frac{\beta-1}{\alpha}}\left(-(k\pi)^2t^{\alpha+\beta}\right)+\frac{\bar{h}_k}{\Gamma(\alpha+1)}t^\alpha E_{\alpha,1+\frac{\beta}{\alpha}, 1+\frac{\beta}{\alpha}}\left(-(k\pi)^2t^{\alpha+\beta}\right).
\end{equation}
We have to note that (\ref{tdfs}) will be a solution to problem (\ref{tdfe})-(\ref{tdfic}) if $\beta>-\{\alpha\}$, which is valid in our case, since $\beta\geq 0$.

In order to find $\bar{h}_k$, we use the condition in (\ref{tdfoc}, which leads to
\begin{equation}\label{tdfss}
\bar{h}_k=\frac{\Gamma(\alpha+1)}{\bar{T}^{\alpha}E_{\alpha,1+\frac{\beta}{\alpha},1+\frac{\beta}{\alpha}}\left(-(k\pi)^2T^{\alpha+\beta}\right)}\left[\bar{\psi}_k-\bar{\varphi}_k\bar{T}^{\alpha-1}E_{\alpha,1+\frac{\beta}{\alpha},1+\frac{\beta-1}{\alpha}}\left(-(k\pi)^2T^{\alpha+\beta}\right)\right].
\end{equation}

Substituting (\ref{tdfs}) and (\ref{tdfss}) into (\ref{tdrs}) we obtain
\begin{equation*}\label{tdrsf}
\bar{u}(t,x)=\sum\limits_{k=1}^\infty \left[\bar{\varphi}_kt^{\alpha-1}E_{\alpha,1+\frac{\beta}{\alpha}, 1+\frac{\beta-1}{\alpha}}\left(-(k\pi)^2t^{\alpha+\beta}\right)+\frac{\bar{h}_k}{\Gamma(\alpha+1)}t^\alpha E_{\alpha,1+\frac{\beta}{\alpha}, 1+\frac{\beta}{\alpha}}\left(-(k\pi)^2t^{\alpha+\beta}\right)\right]\sin k\pi x,
\end{equation*}
\begin{equation*}\label{tdrsfs}
\bar{h}(x)=\sum\limits_{k=1}^\infty \frac{\Gamma(\alpha+1)}{\bar{T}^{\alpha}E_{\alpha,1+\frac{\beta}{\alpha},1+\frac{\beta}{\alpha}}\left(-(k\pi)^2T^{\alpha+\beta}\right)}\left[\bar{\psi}_k-\bar{\varphi}_k\bar{T}^{\alpha-1}E_{\alpha,1+\frac{\beta}{\alpha},1+\frac{\beta-1}{\alpha}}\left(-(k\pi)^2T^{\alpha+\beta}\right)\right]\sin k\pi x.
\end{equation*}

Having an appropriate estimate for the generalized Mittag-Leffler type function appeared in the above solution would in general contribute to the uniform convergence of the above series solutions. However, our estimate in Lemma 1 is not applicable in this case, because the argument of the Mittag-Leffler type function appeared in the solution does not satisfy condition of the Lemma 1. Moreover, appropriate choices of the given conditions will also ensure the convergence of solutions as illustrated the following example solution.

\subsection{Example Solution}
For the sake of the illustration, we consider here the following choices of $\bar{\psi}(x)$ and $\bar{\varphi}(x)$:

$$\bar{\psi}(x)=\sin \pi x, \quad \text{and} \quad \bar{\varphi}(x)=0.$$

Hence, the corresponding solutions are

$$ \bar{u}(t,x)=\dfrac{t^\alpha E_{\alpha,1+\frac{\beta}{\alpha},1+\frac{\beta}{\alpha}}\left(-\pi^2t^{\alpha+\beta}\right)}{T^\alpha E_{\alpha,1+\frac{\beta}{\alpha},1+\frac{\beta}{\alpha}}\left(-\pi^2T^{\alpha+\beta}\right)} \sin \pi x \quad \text{and} \quad \bar{h}(x)= \dfrac{\Gamma(\alpha+1)}{T^\alpha E_{\alpha,1+\frac{\beta}{\alpha},1+\frac{\beta}{\alpha}}\left(-\pi^2T^{\alpha+\beta}\right)} \sin \pi x.$$
The obtained solutions are illustrated graphically in Figures (3) - (4) for $T=1.$ Figure (\ref{fig1}) shows the solution profile at different times and the source function for a fixed value of the fractional order $\alpha$. Similar to the space-degenerate case, the effect of the over-determining condition is clearly seen in the solution profile and in the shape of the source function. However, the solution in this case is decreasing with time and reaching its minimum when $t = T =1$. Whereas, it is increasing with the increase of the fractional order $\alpha$ and reaching it maximum when $\alpha=1$ as shown in Figure (\ref{fig2}). A similar effect of $\alpha$ is also observed in the source term profile as illustrated in Figure (\ref{fig2}) as well.  
\begin{figure}
\includegraphics[height=7cm,width=7cm]{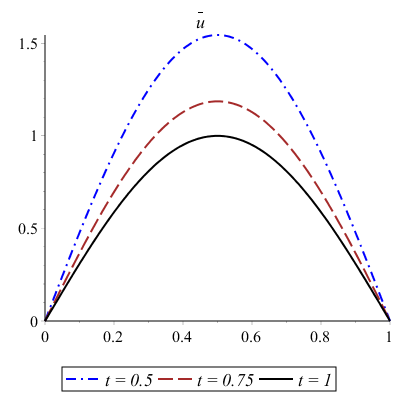}
\includegraphics[height=7cm,width=7cm]{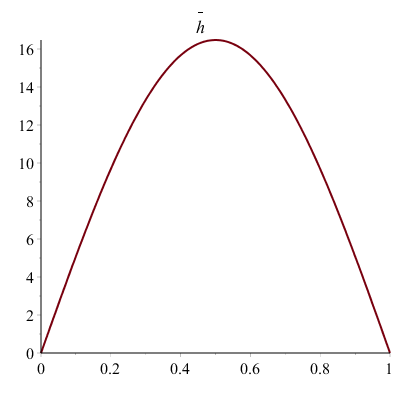}
\caption{Graphs of $\bar{u}(t,x)$ at different times (left) and $\bar{h}(x)$ (right) for $\alpha=\beta=0.5$. }
\label{fig3}
\end{figure}

\begin{figure}
\includegraphics[height=7cm,width=7cm]{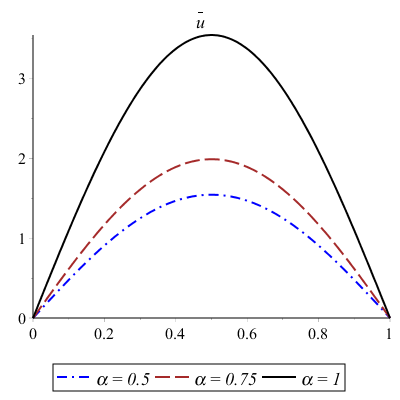}
\includegraphics[height=7cm,width=7cm]{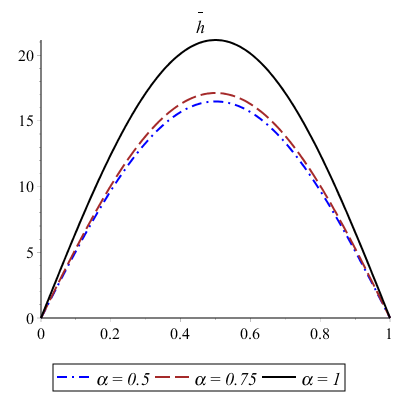}
\caption{Graphs of $\bar{u}(x,t)$ at $t=0.5$ (left) and $\bar{h}(x)$ (right) for different values of $\alpha$, and with $\beta=0.5$ }
\label{fig4}
\end{figure}
\section{Conclusion}

We have considered two inverse source problems for degenerate time-fractional PDEs, which are related to practical problems, namely, heat conduction and groundwater pollution processes. Formal solutions to these problems have been obtained in a form of series expansion using orthogonal basis which are eigenfunctions of self-adjoint spectral problems obtained by considering the corresponding homogeneous equation and using the method of separation of variables. For the first problem, which includes a space-degenerate fractional PDE, solutions were represented in a form of Fourier-Legendre series. Convergence of solutions were obtained by using properties of Legendre polynomials and by imposing certain conditions on the given data. Uniqueness of solution was obtained using the completeness property of the system of Legendre polynomial. The second problem includes a time-degenerate fractional PDE. Solutions to this problem were obtained in a form of Fourier-Sine series which include a generalized Mittag-Leffler type function. In both cases, the obtained results were illustrated by providing example solutions using certain given data. 

Finally, It is worth mentioning here that we have established for the first time a new estimate for a generalized Mittag-Leffler type function as stated in Lemma 1.

\section*{Acknowledgment}
Authors acknowledge financial support from The Research Council (TRC), Oman. This work is funded by TRC under the research agreement no. ORG/SQU/CBS/13/030.

\end{document}